\documentclass[oneside,11pt]{amsart}
\usepackage{amsmath,amsthm}
\usepackage{bbm}
\usepackage{supertabular,longtable}    
\usepackage[T1]{fontenc}  
\usepackage{varwidth}   
\usepackage[foot]{amsaddr} 

\usepackage{lastpage,url} 
\usepackage[left]{eurosym}
\usepackage{fancyhdr} 
\usepackage{amssymb}
\usepackage{epsfig} 
\usepackage{psfrag} 
\usepackage{graphics}
\usepackage{xcolor}  

\usepackage{marginnote}
 
\usepackage{times}

 \usepackage{geometry}
\newenvironment{items}{
\begin{list}{$\alph{item})$}
{\labelwidth100pt \leftmargin30pt \topsep0pt \itemsep2pt \parsep0pt}}
{\end{list}}

\newtheoremstyle{mythm}
{6pt}
{6pt}
{\it}
{}
{\bf}
{.}
{0.5em}
{}

\newtheoremstyle{mydef}
{6pt}
{6pt}
{}
{}
{\bf}
{.}
{.5em}
{}

\newtheoremstyle{myrem}
{6pt}
{6pt}
{}
{}
{\bf}
{.}
{.5em}
{}

 \theoremstyle{mythm}
 \newtheorem{theorem}{Theorem}[section]
\newtheorem{proposition}[theorem]{Proposition}
 \newtheorem{lemma}[theorem]{Lemma}
 \newtheorem{corollary}[theorem]{Corollary}

 \theoremstyle{myrem}

 \newcommand{\g}{\mathfrak{g}}

\renewcommand{\sl}{\mathfrak{sl}}
 \newcommand{\h}{\mathfrak{h}}
 \renewcommand{\t}{\mathfrak{t}}

\renewcommand{\geq}{\geqslant}

\newcommand{\CC}{\mathbbm{C}}

\newcommand{\ZZ}{\mathbbm{Z}}
\newcommand{\NN}{\mathbbm{N}}

\newcommand{\ad}{\mathrm{ad}}
\newcommand{\Ad}{\mathrm{Ad}}

\newcommand{\GL}{\group{GL}}

\newcommand{\SL}{\group{SL}}

\newcommand{\Zen}{\mathrm{Z}}
\newcommand{\Nor}{\mathrm{N}}

\newcommand{\group}{\mathrm}

\renewcommand{\rho}{\varrho}

\renewcommand{\epsilon}{\varepsilon}

\newcommand{\Z}{\mathbb{Z}}

\DeclareMathOperator{\rk}{\mathrm{rk}}
\DeclareMathOperator{\Span}{\mathrm{span}}

\newcommand{\frg}{\mathfrak{g}}
\newcommand{\frh}{\mathfrak{h}}
\newcommand{\fra}{\mathfrak{a}}
\newcommand{\frb}{\mathfrak{b}}
\newcommand{\frc}{\mathfrak{c}}
\newcommand{\frn}{\mathfrak{n}}
\newcommand{\frz}{\mathfrak{z}}
\newcommand{\frm}{\mathfrak{m}}

\newcommand{\frs}{\mathfrak{s}}

\newcommand{\frp}{\mathfrak{p}}
\newcommand{\frq}{\mathfrak{q}}
\newcommand{\frt}{\mathfrak{t}}

\newcommand{\zen}{\mathfrak{z}}

\newcommand{\gl}{\mathfrak{gl}}
\renewcommand{\sl}{\mathfrak{sl}}

\newcommand{\so}{\mathfrak{so}}
\renewcommand{\sp}{\mathfrak{sp}}

\renewcommand{\phi}{\varphi}

\numberwithin{equation}{section}

\newcounter{ithmcount}

\renewcommand{\geq}{\geqslant}
\renewcommand{\leq}{\leqslant}
\newcommand{\diag}{{\rm diag}}

\allowdisplaybreaks

\begin{document}

%

 \author[Dietrich]{Heiko Dietrich}
 \address[H.\ Dietrich, M.\ Origlia]{Monash University,
 School of Mathematics,
 Clayton, VIC 3800, Australia}
 \email{heiko.dietrich@monash.edu}

 \author[Globke]{Wolfgang Globke}
 \address[W.\ Globke]{Faculty of Mathematics,
   University of Vienna,
   1090 Vienna, Austria}
 \email{wolfgang.globke@univie.ac.at}

 \author[Origlia]{Marcos Origlia}
 \email{marcos.origlia@monash.edu}
\dedicatory{Dedicated to the memory of Professor \`E.\ B.\ Vinberg}

\thanks{Dietrich and Origlia were supported by Australian Research Council grant DP190100317; Globke was supported by an Austrian Science Fund FWF grant I 3248.}

\title{A note on \'etale representations from nilpotent orbits}

\begin{abstract}

A linear \'etale representation of a complex algebraic group $G$ is given
by a complex algebraic $G$-module $V$ such that $G$ has a Zariski-open
orbit on $V$ and $\dim G=\dim V$. 
A current line of research investigates which \'etale representations can occur for reductive algebraic groups. Since a complete classification seems out of reach,  it is of interest to find new examples of \'etale representations for such groups. The aim of this note is to describe two classical constructions of Vinberg and of Bala \& Carter for nilpotent orbit classifications in semisimple Lie algebras, and to determine which reductive groups and \'etale representations arise in these constructions. We also explain in detail the relation between these two~constructions.
\end{abstract}

\baselineskip=0.48cm

\maketitle

\vspace*{-1cm}

\section{Introduction}\label{sec_intro}

\noindent Let $G$ be a complex linear algebraic group. A \emph{prehomogeneous module} $(G,\rho,V)$ is a complex algebraic representation $\rho\colon G\to \GL(V)$ such that $V$ is finite-dimensional and $G$ has a Zariski-open orbit in $V$. The points of the open orbit are said to be \emph{in general position} in $V$. In this case,  $V$ is a \emph{prehomogeneous vector space} and  $\dim G\geq\dim V$. If, in addition, $\dim G=\dim V$, then  $(G,\rho,V)$ is an
\emph{\'etale module}, and, accor\-dingly, $\rho$ is an \emph{\'etale representation} of $G$.
Clearly, for \'etale modules, the stabiliser in $G$ is a finite subgroup for any
point in the open orbit.
In terms of Lie algebras $\frg$, an \'etale representation is one where the
action of $\frg$ on a point in general position yields a vector space
isomorphism of $\frg$ and $V$, in particular $\dim\frg=\dim V$, and
the stabiliser subalgebra at a generic point is trivial.
The existence of \'etale representations implies the existence of left-symmetric pro\-ducts on Lie algebras
and thereby also that of left-invariant flat and torsion-free affine connections on
the corresponding Lie groups, see Burde \cite{burde06} for details and additional
references.
Due to this relationship, Lie groups or Lie algebras admitting \'etale representations
are also called \emph{(locally) affinely flat}. 

We are interested in  studying \'etale representations for complex reductive
algebraic groups. It is well-known that many reductive groups do not admit \'etale representations, for example, this is true for semisimple groups. Burde \cite{burde96} shows that if a reductive $G$ with simple commutator subgroup $S$ has an \'etale representation, then  $S=\SL(n,\CC)$ with $n\geq 2$. If $G$ is reductive with 1-dimensional centre and $S$ is not simple but has only pairwise isomorphic simple factors, 
then there are no \'etale modules for $G$, see Burde \& Globke \cite{BG}. A complete classification of \'etale modules for reductive algebraic groups
seems far away, and so the current aim is to find further examples.
Some can be directly obtained by inspecting
Sato \& Kimura's \cite{SK} classifications of prehomogeneous modules for
reductive algebraic groups, cf.~\cite{BG}.
Additional examples with interesting properties were constructed
by Burde et al.~\cite{BGM}.

\subsection*{Results} In the present note, we take a look at the \'etale representations for
reductive algebraic groups arising in the classification of nilpotent orbits in
semisimple Lie algebras, in particular, the classifications of Vinberg \cite{vinberg82} and Bala \& Carter \cite{BC1,BC2}. We show in Proposition \ref{ERfromCA} that 
 these groups  are subject to certain restrictions,
notably that all their simple factors are either special linear or orthogonal groups.
In light of known examples for groups with symplectic groups as simple
factors, e.g.~\cite{BGM}, this shows that \'etale modules of
this type are  a proper subclass of the \'etale modules for general
reductive algebraic~groups. 

 Vinberg's  and Bala \& Carter's  classification methods are very similar, and the second aim of this note is to provide concise descriptions of these methods and to explain how they are related.
In Section \ref{secCA}, we first take a look at Vinberg's construction of
carrier algebras for nilpotent elements in a graded Lie algebra. 
From the classification of simple carrier algebras we
determine the types of reductive groups for which \'etale modules arise by
this method.
In Section \ref{secBC}, we show how Bala \& Carter find minimal Levi
subalgebras for a given nilpotent element in a semisimple Lie algebra, and
explain how it relates to Vinberg's carrier
algebras, see Proposition~\ref{prop:vinberg2BC} and its corollary. In fact, for $\ZZ$-graded algebras, the two approaches coincide.
Gyoja \cite{gyoja96} described how to construct, given a prehomogeneous module $(G,\rho,V)$ for a reductive algebraic group, an
\'etale module $(G',\rho',V')$ for a reductive subgroup $G'\leq G$ and
a quotient module $V'$ of $V$.
In Proposition~\ref{propGyoja}  we show how this generalises Vinberg's and Bala \& Carter's constructions.


\subsection*{Notation}
All Lie algebras $\frg$ we consider here are 
defined over the field of complex numbers.
The centraliser of a subset $X\subseteq \g$ in $\g$ 
is $\zen_\g(X)=\{y\in \g: [y,X]=\{0\}\}$, and the normaliser is
$\frn_\g(X)=\{y\in\g : [y,X]\subseteq \Span_\CC(X)\}$.
An element $x\in\frg$ is  \emph{nilpotent} (or \emph{semisimple}) if its
adjoint representation $\ad(x)$ on $\frg$ is nilpotent (or semisimple). An algebraic group $G$ is \emph{reductive} if its maximal unipotent normal
subgroup is trivial. A Lie algebra $\frg$ is \emph{reductive} if it is
the Lie algebra of a reductive algebraic group. In this case $\frg=\frz\oplus\frs$,
where $\frs$ is the semisimple commutator subalgebra of $\g$ and $\frz=\frz(\g)$ is the centre
of $\frg$. Let $n>0$ be an integer and $\ZZ_n=\{0,\ldots,n-1\}$, or $n=\infty$ and $\ZZ_\infty=\ZZ$. A Lie algebra $\g$ is \emph{$\ZZ_n$-graded} if $\g=\bigoplus\nolimits_{i\in\ZZ_n}\g_i$, where each  $\g_i\leq \g$ is a subspace and  $[\g_i,\g_j]\subseteq \g_{i+j}$ for all $i,j$; here  $\g_k=\g_{k\bmod n}$ for all $k\in\ZZ$.
Note that $\g_0$ is a subalgebra of $\g$.

\section{Vinberg's carrier algebras} \label{secCA}
\noindent Vinberg \cite{vinberg82} studied complex semisimple Lie algebras
graded by an arbitrary abelian group.
However, the first step in his analysis is
to restrict to a subalgebra graded by a cyclic group, so we will only
consider this case.
Let $\g$ be a $\ZZ_n$-graded semisimple Lie algebra, where $n>0$ is an integer or $n=\infty$.
If $n$ is finite, then such a grading is the eigenspace decomposition of a
Lie algebra automorphism of order $n$. If $n=\infty$, then the $\Z$-grading of $\g$ comes from a derivation $\varphi$ that acts as multiplication by $i$ on each $\g_i$.
For semisimple $\g$, this derivation is inner, that is,
$\varphi=\ad(h)$ for a unique \emph{defining element} $h\in\g_0$.

Carrier algebras for $\g$ are constructed as follows. For a nonzero nilpotent $e\in \g_1$ choose an $\sl_2$-triple $(h,e,f)$ where $h\in\g_0$ and $f\in\g_{-1}$; this means $[h,e]=2e$, $[h,f]=-2f$, and $[e,f]=h$. Let $\frt_0$ be a maximal toral subalgebra of the centraliser of $(h,e,f)$ in $\g_0$ and define $\t=\CC h\oplus \frt_0$.
Equivalently, $\t$ is a maximal toral subalgebra of the normaliser of
$\CC e$ in $\g_0$, cf.~\cite[Lemma 30]{dfg15}.
Now let $\lambda\colon \t\to\mathbb{C}$ such that $[t,e]=\lambda(t)e$ for all $t\in\t$, and define the $\Z$-graded algebra $\g(\t,e)$ by
\begin{equation}\label{eqgte}
\g(\t,e)=\bigoplus\nolimits_{k\in\Z} \g(\t,e)_k\quad\text{with}\quad \g(\t,e)_k=\{x\in \g_{k} : [t,x]=k\lambda(t)x\text{ for all }t\in\t\};
\end{equation}
the derived subalgebra of $\g(\t,e)$ is the \emph{carrier algebra} of $e$, denoted
\begin{equation}
\frc(e)=[\g(\t,e),\g(\t,e)].
\label{eq:carrier}
\end{equation}
It is $\Z$-graded with the induced grading; note  that $e\in \frc(e)_1$. This carrier algebra of $e$ is unique up to conjugacy under the adjoint group $G_0$ of $\g_0$; one therefore also speaks of \emph{the} carrier algebra of $e$ in~$\g$. Moreover, two nonzero nilpotent elements of $\g_1$ are $G_0$-conjugate if and only if their carrier algebras are $G_0$-conjugate, which makes carrier algebras a useful tool for classifying nilpotent orbits.
We will not go into the details of this classification, as they are not
required for our purposes here, but an outline is found in Vinberg \cite[Section 4]{vinberg82}.
For details on the classification of nilpotent orbits in \emph{real} semisimple Lie algebras using carrier algebras defined over the real field we refer to Dietrich et al.~\cite{dfg15}.

Vinberg \cite[Theorem 4]{vinberg82} showed that every carrier algebra is semisimple $\Z$-graded with $\frc(e)_k\leq \g_k$ for each $k\in \Z$, and that carrier
algebras are characterized by the following three conditions:\footnote{Vinberg called $\frc(e)$ \emph{locally flat} and \emph{complete} if it satisfies (V1) and (V3), respectively. We avoid this terminology,
  as the semisimple Lie algebra $\frc(e)$ does not admit an
  affinely flat structure; only its reductive subalgebra $\frc(e)_0$ is locally affinely flat.} 

\smallskip

\begin{items}
\item[(V1)]
$\dim\frc(e)_0=\dim\frc(e)_1$;
\item[(V2)]
$\frc(e)$ is normalised by a maximal toral subalgebra of $\frg_0$;
\item[(V3)]
$\frc(e)$ is not a proper subalgebra of a reductive $\ZZ$-graded subalgebra of $\frg$
of the same rank.
\end{items}

\medskip

\noindent Moreover, \cite[Theorem 2]{vinberg82} shows that  $e$ is in \emph{generic} (or \emph{general}) \emph{position} in  $\frc(e)_1$, that is, $[\frc(e)_0,e]=\frc(e)_1$, so (V1) states that the adjoint action of $\frc(e)_0$ on $\frc(e)_1$
yields an \'etale representation for $\frc(e)_0$.

Only property (V1) is intrinsic to $\frc(e)$, whereas (V2) and
(V3) are determined by its embedding in the ambient Lie algebra~$\frg$.
Thus, to describe the Lie algebras that can appear
as carrier algebras for nilpotent elements in semisimple Lie algebras,
one must merely classify $\ZZ$-graded Lie algebras with (V1); we call such an algebra an 
\emph{abstract carrier algebra}.
Every abstract carrier algebra is a direct sum of simple abstract carrier algebras, 
so to describe the possible \'etale modules $(\frc(e)_0,\ad,\frc(e)_1)$
coming from semisimple carrier algebras, it is sufficient to focus on
simple abstract carrier algebras in Lie algebras.
In the next section we follow Djokovi\v{c}'s description \cite{djokovic} (based on work of Vinberg \cite{vinberg82}) of the classification of all simple abstract carrier algebras. Using a different terminology, Bala \& Carter \cite{BC1} have also obtained a classification for the classical case. With these classifications, we determine the following.

\begin{proposition}\label{ERfromCA}
A reductive Lie algebra $\frg_0$ admitting an \'etale representation coming from the adjoint action of a nilpotent element is a direct sum of the degree $0$-components of $j\geq1$ simple abstract carrier algebras.
As such, the semisimple part of $\frg_0$, if non-trivial, has simple factors of type
${\rm A}$ and at most $j$ factors of type ${\rm B}$ or ${\rm D}$.
The centre of $\frg_0$ has dimension $>j$, unless all the simple abstract
carrier algebras involved have weighted Dynkin diagrams of types in
$\{{\rm A}_1,{\rm E}_8^{(11)},{\rm F}_4^{(4)},{\rm G}_2^{(2)}\}$ as defined in \cite[Table II]{djokovic}, in 
which case the centre of $\frg_0$ has dimension $j$.
\end{proposition}

From Burde et al.~\cite{BGM} we know that there exist \'etale representations for
reductive algebraic groups with a simple factor of type ${\rm C}$; this shows the following:

\begin{corollary}
There are \'etale representations for reductive Lie algebras that do not come from
the adjoint action of a nilpotent element.
\end{corollary}
  
\subsection{Simple abstract carrier algebras Lie algebras}
Recall that the grading of a semisimple $\Z$-graded Lie algebra $\g$ with defining element $h\in \g$ is determined as $\g_k=\{x\in\g : [h,x]=kx\}$ for $k\in\Z$. Two $\Z$-graded Lie algebras $\g$ and $\g'$ with defining elements $h$ and $h'$ are isomorphic if there is a Lie algebra isomorphism $\varphi\colon \g\to\g'$ with $\varphi(h)=h'$. Djokovi\v{c} \cite{djokovic} has classified, up to isomorphism, semisimple $\Z$-graded Lie algebras in terms of weighted Dynkin diagrams: Let $\h\leq \g$ be a maximal toral subalgebra containing $h$, with corresponding root system $\Phi$. Let $\Pi$ be a basis of simple roots such that $\alpha(h)\geq 0$ for every $\alpha\in \Pi$.  Let $\Delta(\g)$ be the Dynkin diagram of $\g$ with respect to $\h$, with vertices labeled by $\Pi$, and to each vertex $\alpha\in \Pi$ attach the integer $\alpha(h)$. The resulting weighted Dynkin diagram is denoted $\Delta(\g,h)$. It is  proved  in \cite[Theorem~1]{djokovic} that there is a bijection between (isomorphism classes of) $\Z$-graded semisimple Lie algebras $(\g,h)$ and (isomorphism classes of) weighted Dynkin diagrams $\Delta(\g,h)$.

In the following, let $(\g,h)$ be simple $\ZZ$-graded, and define $\deg\alpha\in\ZZ$ for $\alpha\in \Phi$ by $x_\alpha\in \g_{\deg\alpha}$, where $x_\alpha\in\g$ is a root vector corresponding to  $\alpha$. If $\deg\alpha=k$, then $\alpha(h)x_\alpha=[h,x_\alpha]=kx_\alpha$, hence $\alpha(h)=k$; this shows that $\deg\alpha=\alpha(h)$. If $r_k$ is the number of roots with degree $k$, then $\g$ is an abstract carrier algebra if and only if $\dim\h + r_0= r_1$.
It is shown in \cite[p.\ 374]{djokovic} that if $\g$ is an abstract carrier algebra, then  $\deg\alpha\in\{0,1\}$ for every simple root $\alpha\in\Pi$.
So for the classification it remains to determine the weighted Dynkin diagrams with
weights $\{0,1\}$ such that $\dim\frh_0+r_0=r_1$.
The reductive subalgebra $\frg_0$ is then given by the subdiagram consisting
of the vertices with weight $0$.  To illustrate the method, we include the full proof for type A. To keep the exposition short, for the other types we only describe the results and refer to  \cite[Section 4]{djokovic}, \cite[Section 3]{BC1}, and \cite[p.\ 30]{vinberg82} for more details.

If $\g_0=\h$ is a maximal toral subalgebra, then $r_0=0$ and all labels in the weighted diagram are $1$; one says that $\g$ is \emph{principal}. In this  case the $0$-component of the carrier
algebra is abelian. 

 \bigskip

{\bf Type A.} Let $\g=\sl(n+1,\CC)$. Let the diagonal matrix $h=\diag(\lambda_1,\ldots,\lambda_{n+1})$ be the defining element with  $\lambda_1\geq \ldots\geq \lambda_{n+1}$. Consider the root system  $\Phi=\{\pm (\varepsilon_i-\varepsilon_j) : 1\leq i<j<n\}$ and basis $\Pi=\{\alpha_1,\ldots,\alpha_n\}$ where each $\varepsilon_i$ maps $h$ to $\lambda_i$, and $\alpha_i=\varepsilon_i-\varepsilon_{i+1}$.
Let $k=\lambda_1-\lambda_{n+1}$ and for $i=0,\ldots,k$ let $d_i$ be the number of $\lambda_r$ with $\lambda_r=\lambda_1-i$. A root $\pm(\varepsilon_i-\varepsilon_j)$ has degree $0$ if and only if $\lambda_i=\lambda_j=\lambda_1-r$ for some $r$, and for each~$r$ there are $d_r(d_r-1)$ possibilities for $\varepsilon_i$ and $\varepsilon_j$. This implies that $r_0=\sum_{j=0}^k d_j(d_j-1)$. In a similar way, $r_1=\sum_{j=0}^{k-1} d_jd_{j+1}$, and now a direct calculation shows that  $n+r_0=r_1$ if and only if 
 $$(d_0-d_1)^2+(d_1-d_2)^2+\cdots+(d_{k-1}-d_k)^2+(d_0^2-1)+(d_k^2-1)=0;$$
to see the latter, note that $d_0+\ldots+d_k=n+1$.
In conclusion,  $\g$ is an abstract carrier algebra if and only if
$d_0=\ldots=d_k=1$ and $k=n$, which is equivalent to $\g$ being principal.

 \bigskip

{\bf Type B and D.} Let  $\g=\so(m,\CC)$ be realised as $\g=\{X\in \gl(m,\CC):X^\intercal J=-JX\}$ where $J$ is the matrix that only has $1$s on its anti-diagonal and $0$s elsewhere, and either $m=2n+1$ (with $n\geq 2$) or $m=2n$ (with $n\geq 4$). Write the defining element as $h=\diag(\lambda_1,\ldots,\lambda_m)$ with $\lambda_1\geq \ldots\geq \lambda_m$. Since $hJ=-Jh$,  each $-\lambda_i = \lambda_{m+1-i}$. It has been shown that there are $s\geq 1$ and integers $k_1>\ldots>k_s\geq 0$ such that, as multisets, \begin{eqnarray}\label{eqBD}\{\lambda_1,\ldots,\lambda_m\}=\{ k_i,k_i-1,\ldots,1-k_i,-k_i : 1\leq i\leq s\},
\end{eqnarray}
and  $m=(2k_1+1)+\ldots+(2k_s+1)$; note that $0$ occurs $s$ times in $\{\lambda_1,\ldots,\lambda_m\}$, and $1$ occurs at least $s-1$ times, etc. Conversely, for any such integers $k_1>\ldots>k_s\geq 0$ with  $m=(2k_1+1)+\ldots+(2k_s+1)$ there is a defining element $h$ whose eigenvalues satisfy \eqref{eqBD}. To determine the labeled Dynkin diagrams, one chooses the diagonal matrices in $\g$ as maximal toral subalgebra,  and then the following holds:

\smallskip

If $m=2n+1$, then $s$ is odd and $\lambda_{n+1}=0$. The corresponding simple abstract carrier algebra $B(k_1,\ldots,k_s)$ of type ${\rm B}_n$ has a weighted Dynkin diagram with labels $\lambda_1-\lambda_2,\ldots,\lambda_{n-1}-\lambda_n, \lambda_n$, where $\lambda_n$ is the label of the shorter root, see \cite[Figure 5]{djokovic}. In that figure the last label is given as $2\lambda_n$, which is  a typo; cf.\ \cite[pp.\ 410--412]{BC1}. If $s=1$, then $\{\lambda_1,\ldots,\lambda_m\}=\{n,n-1,\ldots, 1-n,-n\}$ and $B(n)$ is principal. If $s=3$, then $\lambda_{n+1},\lambda_n=0$ and $0<\lambda_{n-1}$, implying that semisimple part of $\g_0$ is a direct sum of algebras of type A. If $s\geq 5$, then $\lambda_{n+1},\lambda_n,\lambda_{n-1}=0$ and that  semisimple part  is a direct sum of algebras of type A and one algebra of type~B.

\smallskip

If $m=2n$, then $s$ is even and $\lambda_n=0$. The corresponding abstract carrier algebra $D(k_1,\ldots,k_s)$ of type ${\rm D}_n$ has a weighted Dynkin diagram with labels $\lambda_1-\lambda_2,\ldots,\lambda_{n-2}-\lambda_{n-1}, \lambda_{n-1},\lambda_{n-1}$, where $\lambda_{n-2}-\lambda_{n-1}$ is the label of the vertex of degree $3$ connected to the two vertices of degree 1 with label $\lambda_{n-1}$, see \cite[Figure 6]{djokovic}.  If $s\geq 6$, then $\lambda_n,\lambda_{n-1},\lambda_{n-2}=0$ and the  semisimple part of $\g_0$ is a direct sum of algebras of type A and one algebra of type~D. If $s=2$ and $k_2>0$, or $s=4$, then that  semisimple part is a direct sum of algebras of type~A; if $s=2$ and $k_2=0$, then $D(n-1,0)$ is principal.

 \bigskip

{\bf Type C.} Let $\g=\sp(2n,\CC)$ be realised as $\g=\{X\in \gl(2n+1,\CC):X^\intercal S=-SX\}$ where $S$ has the identity matrix $I_n$ and the negative $-I_n$ on its anti-diagonal. The simple abstract carrier algebras have the form $C(k_1,\dots,k_s)$ and the construction is similar to those for type B and D: here we can assume the  defining element is $h=\diag(\lambda_1,\ldots,\lambda_n,-\lambda_n,\ldots,-\lambda_1)$ with $\lambda_1\geq \ldots\geq \lambda_n>0$ and, as multisets, $\{\pm \lambda_1,\ldots,\pm \lambda_n\}=\{k_i-\tfrac{1}{2},k_i-\tfrac{3}{2},\ldots,\tfrac{3}{2}-k_i,\tfrac{1}{2}-k_i : 1\leq i\leq s\}$ for some $k_1>k_2>\cdots> k_s>0$ with $n=k_1+\ldots+k_n$.  If one chooses the diagonal matrices in $\g$ as maximal toral subalgebra, then the Dynkin diagram of $C(k_1,\ldots,k_s)$ has  labels $\lambda_1-\lambda_2,\ldots,\lambda_{n-1}-\lambda_n,2\lambda_n$, where $2\lambda_n$ is attached to the longer root, see \cite[Figure 7]{djokovic}. Since $\lambda_n\ne 0$, we have $2\lambda_n=1$, and so the semisimple part of $\g_0$ is a direct sum of algebras of type A. If $s=1$, then $C(n)$ is principal;

\bigskip

{\bf Exceptional types.} A direct calculation yields the abstract carrier algebras $\g$ of exceptional types ${\rm G}_2$, ${\rm F}_4$, ${\rm E}_6$, ${\rm E}_7$, ${\rm E}_8$;  the semisimple part of $\g_0$ is always a sum of Lie algebras of type A, see  \cite[Table~1]{vinberg82}.

 \medskip

{\bf The centre of $\boldsymbol{\frg_0}$.}
Let $\frg_0$ be as before and write $\frg_0=\zen\oplus \frs$ where $\zen$ is the centre and $\frs$ is semisimple. It follows from \cite[p.~19]{vinberg82} that $\dim\zen=\rk\frg_0-\rk\frs=\rk\frg-\rk\frs$. Since $\rk \frs$ equals the number of labels~$0$ in the weighted
Dynkin diagram of $\frg$, the dimension of  $\zen$ equals  the number of labels $1$. For example, if $\g=B(5,2,1)$ with rank $9$, then
$\lambda_1,\ldots,\lambda_9= 5,4,3,2,2,1,1,1,0$, yielding labels $1,1,1,0,1,0,0,1,0$; thus, $\dim\zen=5$. From the above classification, it follows that $\dim\zen=1$ if and only if $\g$ has type ${\rm A}_1$  or if $\g$ is the $\Z$-graded algebra ${\rm E}_8^{(11)}$, ${\rm F}_4^{(4)}$, or ${\rm G}_2^{(2)}$ as defined in  \cite[Table II]{djokovic}.


\section{Bala and Carter's construction}\label{secBC}
 \noindent Bala \& Carter \cite{BC1,BC2} classified the nilpotent orbits in a
complex simple Lie algebra
using a construction very similar to Vinberg's, without the assumption that the
Lie algebras are graded.
Just like Vinberg's construction, this yields an \'etale representation for a
certain reductive subalgebra of $\frg$.
In this section, we review some of these results and show how the approaches by
Vinberg and by Bala \& Carter are~related.

First, we recall a few definitions. Let $\frg=\zen\oplus\frs$ be a reductive Lie algebra with $\frs$ semisimple and $\zen=\zen(\frg)$ the  centre. A \emph{Borel subalgebra} of $\frg$ is a maximal solvable subalgebra of $\frg$, and a subalgebra of $\frg$ is a  \emph{parabolic subalgebra} if it contains a Borel subalgebra of $\frg$. 
Every parabolic subalgebra $\frp$ of $\frg$ is a semidirect product $\frp = \frm\ltimes\frn$
of a nilpotent ideal $\frn$ of $\frp$, all of whose elements are nilpotent, and a
reductive subalgebra $\frm$.
A parabolic subalgebra is  \emph{distinguished} if
$\dim\frn/[\frn,\frn]=\dim\frm$.
Any reductive subalgebra $\frm$ of $\frg$ arising in the above way for some
parabolic subalgebra of $\frg$ is a \emph{Levi subalgebra} in $\frg$.
Its commutator $\frc=[\frm,\frm]$ is 
\emph{semisimple of parabolic type}. Bala \& Carter defined the terms above for semisimple $\frg$, but they carry over
without change to reductive~$\frg$.

For semisimple $\frg$, it is shown in
\cite[Theorem 6.1]{BC2} that the classification of
nilpotent orbits is equivalent to the classification of conjugacy classes of pairs
$(\frc,\frq_{\frc})$, where $\frc$ is semisimple subalgebra of parabolic type in
$\frg$, and $\frq_{\frc}$ is a distinguished parabolic subalgebra of $\frc$: For a nonzero nilpotent $e\in\frg$ with  $\sl_2$-triple $(h,e,f)$ define  $\frg_k=\{x\in\frg : [h,x]=kx\}$ for $k\in \Z$; this furnishes $\g$ with a $\Z$-grading.  With this grading,  $e\in\frg_2$. The element $e$ is \emph{distinguished in $\frg$} if $\ad(e)\colon\frg_0\to\frg_2$ is an isomorphism, that is, if  $e$ is in generic position. If $e$ is not distinguished in $\frg$, then \cite[Propositions 5.3 \& 5.4]{BC2} tell us how to  construct a semisimple subalgebra
$\frc$ of $\frg$  in which $e$ is distinguished: if $\frh_0$ is a maximal toral  subalgebra  of the centraliser of $(h,e,f)$ in $\frg$, then
\begin{equation}\label{eqBC}
 \frm =\zen_{\frg}(\frh_0)\quad\text{and}\quad   \frc=[\frm,\frm]
\end{equation}
are a minimal Levi subalgebra of $\frg$ containing $e$ and a semisimple subalgebra of parabolic type, respectively, such that $e$ is distinguished in $\frc$. The pair corresponding to  $e$ can be chosen to be $(\frc,\frq_\frc)$, where $\frq_\frc$ is the Jacobson-Morovoz parabolic, see \cite[Proposition 4.3]{BC1} and \cite[Theorem~6.1]{BC2}. If $G$ is a semisimple algebraic group with Lie algebra $\frg$, then $\frc$ is determined uniquely up to the action by $\Zen_G(e)$, the centraliser of $e$ in $\Ad_{\frg}(G)$. Since $e$ is distinguished in $\frc$, the $\Z$-grading of its $\sl_2$-triple in $\frc$ yields an  \'etale representation for the adjoint action of the reductive subalgebra $\frc_0$ on the subspace $\frc_2$ by evaluation at $e$. The adjoint
action of $\frg$ integrates to that of $G$, and thus we obtain an \'etale representation of the reductive group with Lie algebra $\frc_0$ on the space $\frc_2$.

The  Borel and parabolic subalgebras of a reductive  $\frg=\frs\oplus\zen$ as above are precisely  $\zen\oplus\frb$ and 	$\zen\oplus\frp$, respectively, with $\frb\leq \frs$  a Borel subalgebra and  $\frp\leq \frs$ parabolic. It follows that 	the Levi subalgebras of $\frg$ are precisely $\zen\oplus\frm$,
	where $\frm$ is a Levi subalgebra of $\frs$; moreover,  $\frm$ is a minimal Levi subalgebra containing $e$ in $\frs$ if and
	only if $\zen\oplus\frm$ is a minimal Levi subalgebra containing $e$ in $\frg$.

\subsection{Relation to carrier algebras}\label{secRelation} We compare the Bala \& Carter construction with Vinberg's carrier algebras. Vinberg starts with a semisimple $\Z_n$-graded Lie algebra $\g=\bigoplus_{i\in\Z_n}\g_i$; recall that we allow $\Z_\infty=\Z$ here. Let $e\in\g_1$ be nonzero nilpotent with $\sl_2$-triple $(h,e,f)$ such that $h\in \g_0$ and $f\in \g_{-1}$, and define $\g(\t,e)$ as in \eqref{eqgte}; as mentioned before, $\frt$ is  a maximal toral subalgebra of the normaliser $\frn_{\frg_0}(e)$ and $\lambda\colon \frt\to\CC$ is defined by $[t,e]=\lambda(t)e$. Let $h_0=\frac{1}{2}h$ define the following $\Z$-graded algebra
\begin{equation}\g(h_0)=\bigoplus\nolimits_{k\in\Z} \g(h_0)_k\quad\text{ with }\quad \g(h_0)_k=\{x\in \g_k : [h_0,x]=kx\};\label{eq:gh0}\end{equation}
note that  $\frt\leq \g(h_0)_0$. It follows from \cite[Lemmas 1 \& 2]{vinberg82} that  $\frg(h_0)$ and $\frg(\frt,e)$ are both reductive. Recall that $\frt=\CC h\oplus \frt_0$, where  $\frt_0$ is a maximal toral subalgebra of $\zen_{\g_0}(h,e,f)$. More precisely:

\begin{lemma}\label{lem:centre_t0}
  We have  $\frt_0=\ker\lambda=\zen(\frg(\frt,e))$ and $\frg(\frt,e)=\zen_{\g(h_0)}(\frt_0)$.
\end{lemma} 
\begin{proof}
  Write $\zen=\zen(\frg(\frt,e))$.  Recall that  $\frt=\CC h\oplus \frt_0$, so $\ker\lambda=\frt_0$ follows from $[h,e]=2e$. Clearly, if $t\in\ker\lambda$, then $[t,y]=0$ for each $y\in \g(\t,e)_k$, so $t\in\zen$. Since $\zen$ commutes with the defining element of  $\frg(\frt,e)$, we have  $\zen\leq\frg(\frt,e)_0$. Thus, $\zen\leq\frn_{\frg_0}(e)$, and therefore  $\zen\leq \frt$, This implies $\zen\leq \ker\lambda$, hence $\zen=\ker\lambda$. Suppose $x\in\g(h_0)$ centralises $\frt_0$ and write $x=\bigoplus_{k\in\Z} x_k$ with each $x_k\in \g(h_0)_k$. Since $t_0\leq \g(h_0)_0$, it follows from $0=[x,\frt_0]=\bigoplus_{k\in\Z} [x_k,\frt_0]$ that each $x_k$ centralises $\frt_0$. By assumption, $[h,x_k]=2kx_k$, so $x_k\in \g(\frt,e)_k$, and hence $\zen_{\g(h_0)}(\frt_0)\leq \g(\frt,e)$. Conversely, if $x\in\frg(\frt,e)_k$, then $[h,x]=2kx$ and, if $t\in \frt_0$, then $[t,x]=k\lambda(t)x=0$ since $\frt_0=\ker\lambda$. Thus, $x\in \zen_{\g(h_0)}(\frt_0)_k$.
\end{proof}

The next proposition shows that Vinberg's construction \eqref{eqgte}
of $\frg(\frt,e)$ and its carrier algebra is the same as applying
Bala \& Carter's approach \eqref{eqBC} to  the $\ZZ$-graded Lie algebra
$\frg(h_0)$.
Below, let $G$ be the semisimple algebraic group with Lie algebra
$\frg(h_0)$.
The conjugacy up to the centraliser $\Zen_G(e)$ reflects the freedom in choosing an
$\sl_2$-triple $(h,e,f)$ for a given nonzero nilpotent element $e$.

\begin{proposition}\label{prop:vinberg2BC}
Let $\frg$ be a $\ZZ_n$-graded complex semisimple Lie algebra, where
$n\in\NN\cup\{\infty\}$, and let $e$, $h_0$, and $\frt$ be as above.
Then $\frg(\frt,e)$ is a minimal Levi subalgebra of $\frg(h_0)$
containing $e$ and,  up to $\Zen_G(e)$-conjugacy, the carrier subalgebra $\frc(e)=[\frg(\frt,e),\frg(\frt,e)]$ is the unique semisimple
subalgebra of parabolic type in $\frg(h_0)$ in which $e$ is distinguished.
\end{proposition}
\begin{proof} Note that $h$ stabilizes each $\g_k$, and  $\frg(h_0)_k$ is the intersection of $\g_k$ with the $2k$-eigenspace of~$h$.  Lemma \ref{lem:centre_t0} shows that  $\frg(\frt,e)\leq \frg(h_0)$, and the $\Z$-gradings of both algebras are determined by the eigenvalues of $\ad(h_0)$. The semisimple part  $\frs$ of  $\frg(h_0)$ is a semisimple ideal in $\frg(h_0)$ containing $(h,e,f)$;  let $\fra$ be the subalgebra generated by $\{h,e,f\}$.  Note that for every subset $X\subseteq \frg(h_0)$ we have 
\[  
\zen_{\frg(h_0)}(X) = \zen(\frg(h_0))\oplus\zen_{\frs}(X).
\tag{$*$}\label{*}
\]
We claim that $\frt_0$ is a maximal toral subalgebra of $\zen_{\frg(h_0)}(\fra)$: recall that $\frt_0$ is defined as a maximal toral subalgebra of $\zen_{\g_0}(\fra)$, which is reductive by \cite[p.\ 21]{vinberg82}.
Since $\frt_0\leq \g(h_0)_0\leq \g_0$, we know that  $\frt_0$ is also a maximal toral subalgebra in $\zen_{\frg(h_0)_0}(\fra)$.  On the other hand, we have $\zen_{\frg(h_0)_0}(\fra)=\zen_{\frg(h_0)}(\fra)$ because elements of degree $\neq0$ in $\frg(h_0)$ do not commute with the defining element $h_0\in\fra$; thus, $\frt_0\leq \zen_{\frg(h_0)}(\fra)$ is a maximal toral subalgebra. We can write  $\frt_0 = \zen(\frg(h_0)) \oplus \frt_0'$, where $\frt_0'$ is a maximal toral subalgebra of $\zen_{\frs}(\fra)$, and so for every subset $X\subseteq\frg(h_0)$ we have
\[
\zen_{X}(\frt_0)=\zen_{X}(\frt_0'). 
\tag{$**$}\label{**}
\]
 The construction in \eqref{eqBC} shows that  $\frm'=\zen_{\frs}(\frt_0')$ is a minimal Levi subalgebra of $\frs$ containing~$e$, so  $\frm=\zen(\frg(h_0))\oplus\frm'$ is a minimal Levi subalgebra of $\frg(h_0)$ containing~$e$. Now \eqref{*}, \eqref{**}, and Lemma~\ref{lem:centre_t0}~show
\[
\frm = \zen(\frg(h_0))\oplus\zen_{\frs}(\frt_0')
= \zen_{\frg(h_0)}(\frt_0')
= \zen_{\frg(h_0)}(\frt_0)
= \frg(\frt,e),
\]
so $\frg(\frt,e)$ is a minimal Levi subalgebra in $\frg(h_0)$
containing $e$. The construction in \eqref{eqBC} shows that $e$ is distinguished in $[\frm,\frm]$, and the latter is semisimple of parabolic type.
Since $[\g(\frt,e),\g(\frt,e)]=[\frm,\frm]$
is the carrier algebra, the claim follows; \cite[Proposition~5.3]{BC2} shows uniqueness up to $\Zen_G(e)$-conjugacy.
\end{proof}

\enlargethispage{3ex}
For $\ZZ$-graded semisimple Lie algebras $\frg$, we have
$\frg=\frg(h_0)$ where  $h=2h_0$ is the defining element, see \cite[Remark 33]{dfg15},
so the two approaches by Vinberg and Bala \& Carter coincide.

\begin{corollary}\label{cor:vinberg2BC}
If $\frg$ is a $\ZZ$-graded complex semisimple Lie algebra, then up to $\Zen_G(e)$-conjugacy, the subalgebra $\frg(\frt,e)$ obtained by Vinberg's construction and the
subalgebra $\frm$ obtained by Bala \& Carter's construction coincide.
\end{corollary}

Even in the situation where $\frg$ is given without a grading and $e$ is a nonzero
nilpotent element in $\frg$, a choice of $h_0$ induces a $\ZZ$-grading on
$\frg$ to which Vinberg's approach can be applied; this is then
equivalent to Bala \& Carter's approach; note that Bala \& Carter use the
element $h$ rather than $h_0=\frac{1}{2}h$ to define their grading, which
leads to an additional factor two in the degrees.

\section{Gyoja's construction}\label{secGyoja}
\noindent Gyoja \cite{gyoja96} described constructions of \'etale modules out of a given prehomogeneous module.  Let $G$ be a complex reductive algebraic group with algebraic representation
$\rho\colon G\to\GL(V)$ on a finite-dimensional complex vector space $V$ such that
$(G,\rho,V)$ is a prehomogeneous module. Let $v\in V$ be the point in generic position. The construction in \cite[Theorem A]{gyoja96} proceeds as follows: Let $G_v$ be the stabilizer of $v$ in $G$ with Cartan subgroup $T\leq G_v$, define $G'=\Nor_G(T)/T$, and let $V'=V^T$
be the set of fixed points in $V$ under $T$; then $V'$ is an \'etale module for the induced action of $G'$. Arising from a normaliser of a torus, $G'$ is a reductive algebraic group.  
The second construction  \cite[Theorem~B]{gyoja96}  yields  a procedure
to obtain a super-\'etale module from an \'etale module. (This means that the stabiliser of the point in generic position is trivial and not just finite.) Given an \'etale module $(G,\rho,V)$
with  $v\in V$ in generic position and stabilizer $G_v$, choose  $1\neq h\in G_v$, and let $G''=\Nor_{G}(h)$ and $V''=V^h$, the set of fixed points for $h$ in $V$. Then $V''$ is an \'etale module for the induced action of $G''$, and $|G''_v|<|G_v|$. Since $\langle h\rangle$ is finite,  $G''$ is also reductive. After finitely many iterations (with  $G''$ instead of $G$), one obtains a super-\'etale module.

\subsection{Relation to carrier algebras}\label{secRelation2}
Gyoja's construction  was formulated for groups, but can just as well be
formulated for the corresponding Lie algebras. We show that this covers some of Vinberg's constructions. For this let $\g$ be a semisimple Lie algebra with nonzero nilpotent $e\in \g$, let $(h,e,f)$ be an $\sl_2$-triple in $\g$, and furnish $\g$ with the $\Z$-grading induced by $\ad(h_0)$, where $h_0=\tfrac{1}{2}h$. 

\begin{proposition}\label{propGyoja}
Gyoja's construction, applied to the reductive Lie algebra $\g_0$, the nilpotent element $e$, and the adjoint $\frg_0$-module $(\ad,\frg_1)$, produces Vinberg's \'etale representation associated with $e\in\g$.
\end{proposition}
\begin{proof}The stabiliser algebra of $e$ is $\zen_{\g_0}(e)$. We have shown that  $\frt_0=\ker\lambda$ (as introduced in Lemma~\ref{lem:centre_t0}) is a maximal toral subalgebra of $\zen_{\frg_0}(\fra)$, where $\fra$ is the subalgebra spanned by $(h,e,f)$. By \cite[p.~21]{vinberg82}, we have $\zen_{\frg_0}(\fra)=\zen_{\frg_0}(e,h)$, and since $h=2h_0\in \frg_0$, it follows that $\zen_{\frg_0}(\fra)=\zen_{\frg_0}(e)$. Thus, a maximal toral subalgebra of $\zen_{\frg_0}(e)$ is $\frt_0=\ker\lambda$, which takes the place of Gyoja's~$\frt$. Now Lemma \ref{lem:centre_t0} shows that the fixed point set of $\frt_0$ in $\frg_1$ is $V'=\zen_{\frg}(\frt_0)\cap\frg_1=\frg(\frt,e)_1$, with $\frt=\CC h\oplus \frt_0$; moreover, $\zen_{\frg_0}(\frt_0)=\frg(\frt,e)_0$ is reductive with centre $\frt_0$. Hence, $\frg'=\zen_{\frg_0}(\frt_0)/\frt_0$ satisfies $\frg'\cong[\frg(\frt,e)_0,\frg(\frt,e)_0]$, so Gyoja's $\frg'$ is the
  $0$-component of the carrier algebra of $e$ in $\g$. Lastly, $V'=\g(\t,e)_1=[\g(\t,e),\g(\t,e)]_1$ is the 1-component of that carrier algebra: this follows since the centre of the reductive $\g(\t,e)$ is contained in $\g(\t,e)_0$.
\end{proof}

Gyoja's result encapsulates what is interesting to us in Vinberg's and
Bala \& Carter's theory: the focus on the \'etale action of the reductive subalgebra
of degree $0$ on the subspace of degree $1$, ignoring the subspaces of higher
degree.

\enlargethispage{1ex}

{
}
\end{document}